\documentclass[amsart]{article}
\usepackage{tikz-cd}
\usepackage{mycommands}
\usepackage{mathdots}
\usepackage{apptools}
\AtAppendix{\counterwithin{thm}{section}}

\author{Dmitry Vaintrob}

\renewcommand{\t}[1]{\mathrm{#1}}
\newcommand{\Shf}{\mathrm{Shf}}
\renewcommand{\top}{\mathrm{Top}}
\renewcommand{\phi}{\varphi}

\renewcommand{\vec}{\overrightarrow}

\renewcommand{\F}{\mathcal{F}}
\renewcommand{\X}{\mathcal{X}}
\newcommand{\Y}{\mathcal{Y}}
\newcommand{\I}{\mathcal{I}}

\renewcommand{\L}{\mathcal{L}}

\renewcommand{\log}{\mathrm{log}}

\newcommand{\fld}{\mathrm{FLD}}
\newcommand{\FLD}{\mathrm{FLD}}

\newcommand{\LD}{\mathrm{LD}}
\newcommand{\DMK}{\mathrm{DMK}}

\newcommand{\dlog}{\operatorname{dlog}}

\newcommand{\forg}{\mathtt{forg}}

\newcommand{\comp}{\operatorname{comp}}
\newcommand{\fldlog}{{\mathrm{FLC}}}
\newcommand{\cech}{\text{\v{C}ech}}
\newcommand{\sing}{\text{sing}}

\title{Formality of little disks and algebraic geometry}

\begin{document}

\maketitle
\abstract{We construct a canonical chain of formality quasiisomorphisms for the operad of chains on framed little disks and the operad of chains on little disks. The construction is done in terms of logarithmic algebraic geometry and is remarkable for being rational (and indeed definable integrally) in de Rham cohomology. 
}

\newpage
\tableofcontents
\newpage

\section{Introduction}
Let $\fld$ be the operad of framed little disks. The author showed in \cite{ffc} that $\fld$ is equivalent (via a chain of quasiisomorphisms of topological operads) to the topological realization of an operad $\fldlog$ (``framed little curves'') of operads in logarithmic algebraic geometry in the sense of Kato. In this paper we write down explicitly the log spaces and maps involved, and show that these spaces have an incredibly well-behaved Hodge splitting property (i.e.\ quasiisomorphism between cohomology of forms and de Rham cohomology), which we call the \emph{proper acyclicity} property, something that is only possible in the log geometry context. From this property we deduce an explicit formality splitting quasiisomorphism between dg operads $H_*(\fldlog, \cc)$ and $C_*(\fldlog,\cc)$. We deduce similar splittings for the operad $LD$ of little disks as a consequence. Our construction is canonical (in an appropriate derived sense) and explicit. It does not depend on a choice of an associator or a Frobenius element in the Grothendieck-Teichm\"uller group, or a choice of obstruction-theoretic splitting. Its algebro-geometric provenance also gives it automatic compatibility with a number of structures. In particular this splitting is rational (and indeed definable integrally) for the ``de Rham'' rational structure on the operad $\fld$ (the de Rham structure on the complex $C^*(\fld)$ as well as on $C^*(LD)$ is part of a larger known derived mixed Hodge structure on these dg operads: see for exampe \cite{ch_mixed} for a definition in terms of the Grothendieck-Teichm\"uller group, or \cite{ffc} for a log algebro-geometric definition.) No explicit splitting that has previously been constructed was known to be rational in any lattice, though rational splittings have been known to exist by certain standard lifting theorems. Our rationality implies compatibility of this formality quasiisomorphism with the theory of derived vertex algebras with rational coefficients, and (via results of Vallette and Drummond-Cole) with rational structure on the genus $0$ Deligne-Mumford-Knutsen operad. In the upcoming paper \cite{associators}, the author shows that this formality isomorphism has an interesting and new deformation in the presence of monoidal structure determined by an associator.

\subsection{Relation to previous work} A formality splitting for the operad of little disks was first proven to exist by Dmitry Tamarkin using path integral ideas of Kontsevich in \cite{tamarkin_formality}, then using Drinfeld associators in \cite{tamarkin_quantization}. Tamarkin's proof was extended to the operad of framed little disks independently by Giansiracusa and Salvatore in \cite{gs_formality} and by Severa in \cite{severa}. A differently flavored proof, using Grothendieck-Teichm\"uller action, was given in \cite{cirici-horel}. A splitting similar to the one constructed here was sketched out in an unpublished short letter of Beilinson to Kontsevich, \cite{beilinson_let}, and it would not be wrong to say that the present paper is a formalization of an idea of Beilinson.


\subsection{Idea of proof and structure of paper}
For convenience we construct the formality quasiisomorphism in a dual context, i.e., for the cooperad $C^*(\fld,\cc)$ of
cochains instead of chains. We construct the formality quasiisomorphism via a \emph{log de Rham} comparison, $$H^*(\fld, \cc)\to C_{dR}^*(\fldlog)$$ for $\fldlog$ our log algebro-geometric model for framed little disks. Here the formality follows from the following two observations, each of which follow from the \emph{proper acyclicity} property of the log spaces $\fld_n^{log}$: \begin{enumerate}
\item For each $n$, the log space of operations $\fld_n^{log}$ satisfies  $$H^p\Omega^q(\fld_n) \cong \begin{cases}H^q(\fld_n, \cc), & p = 0\\ 0, & p\ge 1. 
\end{cases}.$$
\item The de Rham differentials $d:H^0\Omega^q\to H^0\Omega^{q+1}$ are zero. 
\end{enumerate}
In other words, the Hodge to de Rham spectral sequence degenerates at the $E1$ term and only has one row, which implies formality of the cooperad $C^*_{dR}(\fldlog).$ We then compare $C^*_{dR}(\fldlog)$ with $C^*(\fld,\cc)$ by a topological argument: the so-called ``Kato-Nakayama'' topological realization $\fldlog^{top}$ of $\fldlog$ is related by a chain of homotopy equivalences of topological operads to $\fld$, and $C^*_{dR}(\fldlog)$ is equivalent to the topological, i.e., Betti cohomology $C^*(\fldlog^{top}, \cc)$ by a Betti-to-de Rham comparison.

The central sections of this paper are Section \ref{sec:log_disks}, where we introduce the relevant logarithmic geometry objects and Section \ref{sec:hodge_thy}, where we deduce formality from their Hodge theoretic properties. It is our goal whenever possible, and especially in these central sections, to be maximally explicit about the algebra and topology, and to assume no prior knowledge of logarithmic geometry. Section \ref{sec:log_disks} opens with a brief ``minimalist'' introduction into the category of ``normal-crossings'' logarithmic schemes (a convenient subcategory sufficient for our purposes) in Section \ref{sec:log_intro}, and proceeds to define the spaces and morphisms comprising the operad $\fldlog$. 

In Section \ref{sec:hodge_thy}, we discuss de Rham cohomology and Hodge theory of logarithmic varieties, and introduce the proper acyclicity property. We prove that the spaces comprising the operad $\fldlog$ are proper acyclic. In section \ref{sec:ld}, we describe a model for the operad of little disks as a fiber of a map of (proper acyclic) log operads, and deduce an explicit formality isomorphism for this operad; this part makes use of certain supplementary structures on the operad $\fldlog$ described in section \ref{sec:forgetful}. 

\subsection{The main theorems}\label{sec:results}
The main result of this paper is a canonical formality isomorphism for the (chain complex) operads of framed and unframed little disks.
Let $LD, FLD$ be the operads of unframed and framed little disks, respectively. We define the $\infty$-category of operads of chain complexes over $\cc$ to be the $\infty$-category localization of the ordinary category of operads of chain complexes by maps which are levelwise quasiisomorphisms.\footnote{By results of H.\ Chu, D.-C.\ Cisinski, G.\ Heuts, V.\ Hinich, I.\ Moerdijk and others the $\infty$-category underlying this model category is equivalent to Lurie's category of (single-colored) $\infty$-operads.} 
\begin{thm}
There are formality isomorphisms in the $\infty$-category of operads of chain complexes
\[H_*(FLD, \cc)\to C_*(FLD, \cc)\] and
\[H_*(LD, \cc)\to C_*(LD, \cc),\]
which are canonical up to a contractible space of choices and compatible with the map of operads $LD\to FLD$, as well as the comultiplication on each homology chain complex. 
\end{thm}

The formality quasiisomorphism is a consequence of the following two results, which explain that (in a certain motivic sense), the canonical formality can be defined $\qq$ and can even be lifted (as a quasiisomorphism up to torsion) to have $\zz$ coefficients.

\begin{thm}[Proper acyclicity and formality]\label{thm:PALog-main}
There is a symmetric monoidal category $\t{PALog}$ of so-called proper acyclic log varieties over $\qq$ (a full subcategory of the category of the category $\t{LogSch}_\qq$ of log varieties over $\qq$) such that the Kato-Nakayama topological realization functor \[KN:\t{LogSch}_\qq\to \t{Top}\] restricted to $\t{PALog}$ has the following formality properties:
\begin{enumerate}
\item The Betti cochains functor with coefficients in $\cc,$ defined as the composition \[C^*_{\mathrm{Betti}}:\t{PALog}\xrightarrow{KN} \t{Top}\xrightarrow{C^*_\cc} D\t{Vect}\] from $\t{PALog}$ to the $\infty$-category of complexes of vector spaces is canonically formal in a symmetric monoidal sense. In other words, it is equivalent as a symmetric monoidal functor to the Betti homology functor \[H^*_{\mathrm{Betti}}:\t{PALog}\xrightarrow{KN} \t{Top}\xrightarrow{H^*_\cc} D\t{Vect}.\]
\item For $\Gamma\Omega^k:\t{LogSch}_\qq\to \T{Vect}_\qq$ the functor of global sections \emph{over $\qq$} of log $k$-forms, the $k$th homology functor $H^k_{\mathrm{Betti}}$ is canonically equivalent to the $\cc$-basechange $\Gamma\Omega^k\otimes_\qq\cc$, and this identification is symmetric monoidal in the variety when viewed as a natural transformation $\Gamma\Omega^*\otimes_\qq\cc\to H^*_{\t{Betti}}.$
\end{enumerate}
\end{thm}
\begin{thm}\label{thm:fldlog-main}
\begin{enumerate}
\item There is a non-unital operad $\fldlog$ of objects in $\t{PALog} = \t{PALog}_\qq,$ which we call the operad of ``framed little (log) curves'' whose topological realization is related by a canonical chain of topoogical operad equivalences to $FLD.$
\item There is a map of operads in $\t{PALog}$ \[\fldlog\to \t{Comm}_{\pt_{log}}\] (where the operad $\t{Comm}_{\pt_{log}}$ is the operad responsible for ``commutative log point equivariant monoids'') which is homotopy equivalent to the operad $LD$ of little disks after taking homotopy fibers of the map of topological realizations. Furthermore, the corresponding map on graded homology rings is cofibrant when viewed as a map of cdga's.
\item Moreover, the operads $\fldlog, \t{Comm}_{\pt_{log}}$ and the map between them are defined as operads of log varieties over $\zz.$
\end{enumerate}
\end{thm}
Together the two theorems imply formality for $FLD$ and (using the cofibrancy result), for $LD.$ The operads and quasiisomorphisms we construct are a priori not unital, but ``unitizable'', i.e., equivalent to unital operads; by standard homotopy-theoretic arguments, the non-unital formality quasiisomorphisms imply also formality of the corresponding unital operads: see for example Roig, \cite{roig}.

\subsection{Acknowledgments}
The author is grateful for discussions with Clemens Koppensteiner and Mattia Talpo, which are responsible for much of the log geometric content of this paper. Many of the weight filtration ideas in Section \ref{sec:operad_pa} were explained to the author by Joana Cirici. The idea for giving the explicit definition of the log composition morphisms in section \ref{sec:fld_def} was suggested an anonymous reviewer of the paper \cite{ffc}, to whom the author is indebted. A number of discussions were instrumental for the gestation of this paper, and the author would like to in particular thank Pavel Etingof, Akhil Mathew, Agust\'i Roig and Dmitry Tamarkin.

\section{Logarithmic geometry and the operad of framed little disks}\label{sec:log_disks}
\subsection{The category of logarithmic schemes}\label{sec:log_intro} See \cite{ogus_logbook} for an in-depth development of Kato's logarithmic geometry, and \cite{talpo_log} for an informal introduction; we give a minimalistic picture of the relevant theory here. We work over a characteristic-zero field $k,$ which we will mostly take to be either $\qq$ or $\cc.$ Given a scheme $X$, a log structure on $X$ is an \'etale sheaf of monoids $\M/X$ with a certain multiplicative relationship to the sheaf of functions $\oo_X.$ Schemes with log structure form a category $\t{LogSch}$ with behavior analogous to the category of schemes. In particular, if $X$ is an algebraic variety then one can turn it into a log scheme by taking ``trivial log structure'' on $X$ (which we will abusively still denote $X\in \t{LogSch}$); this realizes the category $\t{Sch}$ of schemes as a full subcategory of the category of log schemes. In the other direction, we can define a functor taking a log scheme $\X$ corresponding to a log structure on the scheme $X$ to the underlying space, $$\X\mapsto \mathring{\X}: = X.$$ The forgetful functor $\X\mapsto \mathring{\X}:\t{LogSch}\to \t{Sch}$ is left adjoint to the canonical embedding of $\t{Sch}$ in $\t{LogSch},$ with unit of the adjunction given by natural transformation $$\pi_\X: \X\to \mathring{\X}$$ for $\X\in \t{LogSch},$ which we call projection to the underlying scheme.
\begin{rmk}
All log schemes we work with will be fine and saturated. Moreover, they will be of \emph{normal-crossings type} (though we treat a larger class of log schemes in the Appendix, \ref{appendix}). 
\end{rmk}
\subsection*{Omega Smoothness.}
Normal-crossings log schemes have a special smoothness property which we call $\Omega$-smoothness. $\Omega$-smooth log schems are the class of schemes for which Kato and Nakayama proved their Betti-de~Rham correspondence. In the language of Ogus \cite{ogus_logbook}, they can be characterized as log schemes which are smooth when endowed with tautological idealized structure. See Section \ref{sec:dsmooth} of the Appendix for details.

\subsection*{Normal-crossings log schemes}
We now give a suite of definitions and results which will provide sufficient background on log schemes of normal-crossings type to explicitly define and work with our operad $\fldlog$. Proofs of all of these statements can be found in \cite{ogus_logbook}. 

\begin{enumerate}
\item Given a smooth scheme $X$ together with a normal-crossings divisor $D\subset X$, there is a log scheme $(X, D)_{log}$ with underlying space $X$. Such a log scheme is in particular \emph{smooth}.
\item To a log scheme $\X$ with underlying scheme $X$ (resp., a map of log schemes $\X\to \Y$), we can associate a bundle $$\Omega_\X$$ over $X$ of log differentials (resp., $\Omega_{\X/\Y}$ of relative log differentials). We define $$\Omega^k_\X : = \Lambda^k_{/\oo_X}\Omega,$$ in particular $\Omega^0_\X = \oo_X.$ For $\X = X$ (trivial log structure), $\Omega_\X = \Omega_X.$ For a log scheme of type $\X = (X, D)_{log},$ we have $\Omega_\X = \Omega_X(D, \t{log}),$ the bundle of rational differentials generated locally by $\dlog f$ for $f$ functions with no zeroes or poles outside $D.$
\item We define the log tangent bundle $T_\X := \Omega^\vee_\X$ to be dual to the cotangent bundle.
\item There is a differential $d:\Omega^k_\X\to \Omega^{k+1}_\X$ which is a map of sheaves of vector spaces over $X$ (similarly to the non-logarithmic context). The de Rham cohomology $$H^*_{dR}(\X)$$ of a log scheme $\X$ is defined as the hypercohomology of the complex of sheaves $(\Omega^*_\X,d)$ on $X$.
\item To a scheme $X$ and a line bundle $L$ on $X$ one associates a log scheme $(X, L)_{log}$ with underlying scheme $X$, with monoid $\M_{(X, L)}$ given by all homogeneous functions on the total space of $L$. For $X = \pt, L = \oo_{pt}$ the trivial line bundle, we define the ``log point'' $$\pt_{log} : = (\pt, \oo)_{log}.$$ Schemes of the form $(X, L)_{log}$ (for $X$ smooth) are $\Omega$-smooth (See Appendix, \ref{sec:dsmooth}). Moreover the projection to the underlying scheme $(X, L)_{log}\to X$ is an $\Omega$-smooth map.
\item $\Omega(X, L)_{log}$ is the sheaf whose sections over $U\subset X$ are $\gg_m$-equivariant differentials of the $\gg_m$-torsor $\gg L$ on $U$ given by removing the zero section of $L.$
\item There is a useful intuition for the log scheme $(X, L)_{log},$ which has a geometric correlate in terms of Kato-Nakayama spaces (introduced in the next section). Namely, $(X, L)_{log}$ can be thought of as a ``zeroth order logarithmic neighborhood'' of the zero section in the total space of $L$. Indeed, functions on $(X, L)_{log}$ are simply functions on $X$, while one-forms on $(X, L)_{log}$ are locally generated by restrictions to the zero-section $X\subset L$ of regular one-forms on $L$ (i.e., one-forms on $X$) and ``residue terms'' of singular one-forms on $L$ with first order singularity along the zero section.
\item Log schemes have a notion of base change which is compatible with base change of underlying schemes, and base change with respect to an $\Omega$-smooth map preserves $\Omega$-smoothness and induces a pullback diagram of sheaves of (pulled back) log tangent bundles in a standard sense.
\item Let $X$ be a smooth scheme with a collection of line bundles $L_1, L_2, \dots, L_d.$ Write $\underline{L} : = (L_1,L_2,\dots, L_d)$ for the tuple of bundles (this is given by the same data as a $\gg_m^d$-principal bundle on $X$). We define $(X, \underline{L})_{log}$ for the iterated fiber product of the $(X, L_i)_{log}$ over $X$. This is an $\Omega$-smooth log scheme, and its projection to the underlying space $X$ is $\Omega$-smooth with fibers $\pt_{log}^d$.
\item More generally given a smooth scheme $X$ with a tuple $\underline{L} = (L_1,\dots, L_d)$ of line bundles as above and a normal-crossings divisor $D\subset X$, write $$(X, D, \underline{L}) : = (X, \underline{L})_{log}\times_X (X,D)_{log}.$$ This is the base change of the smooth log scheme $(X, D)_{log}$ under an $\Omega$-smooth map, hence is $\Omega$-smooth. 
\item Schemes \'etale locally of the type $(X, D, \underline{L})_{log}$ as above are called log schemes of normal-crossings type. Schemes \emph{Zariski} locally of the type $(X, D, \underline{L})_{log}$ with $D\subset X$ a strict normal-crossings divisor are called of strict normal crossings type. All log schemes we deal with will be of strict normal-crossings type.
\item To an $\Omega$-smooth log scheme $\X$, one associates two numbers: the log dimension, defined as the rank of the log tangent bundle $\Omega_\X$ and the geometric dimension, defined as the dimension of the underlying scheme $X$. There is also a log fiber dimension, defined as the difference of the two. The scheme $(X, D, \underline{L})_{log}$ has geometric dimension $n$ and log dimension $n+d,$ where $n = \dim(X)$ and $d$ is the number of line bundles in the tuple $\underline{L}.$
\item Given two log schemes $\X, \X'$ both with underlying scheme $X,$ we say that a \emph{map of log structures} is a map $\X\to \X'$ over the identity map of underlying schemes.
\item Given a map of schemes $f:X\to Y$ and a log structure $\Y$ on $Y$, there is a ``pullback log structure'' $f^*(\Y)$ with underlying scheme $X$ and canonical map $f^*(\Y)\to \Y$. Any map $\X\to \Y$ with map of underlying schemes $f$ factors uniquely through this map as $\X\to f^*\Y\to \Y,$ with the map $\X\to f^*\Y$ a map of log structures.
\item Assume $\X = (X, D, \underline{L})$ and $\X' = (X, D', \underline{L}')$ are two normal-crossings log structure with the same underlying scheme $X$. Then a map of log structures $$\X\to \X'$$ is classified by the following data.
$$\t{Maps}_X(\X, \X') = \begin{cases} \emptyset, & D'\nsubseteq D\\
\{\alpha:\aa \underline{L}\to \aa \underline{L}'\mid \alpha\t{ homogeneous}\}, & D'\subseteq D.
\end{cases}
$$
Here $\aa\underline{L}$ is the total space of the vector bundle $L_1\oplus \dots \oplus L_d$ and a map between two such bundles is homogeneous if elements of pure degree $(m_1, m_2,\dots, m_d)$ map to elements of pure degree $(m'_1,m'_2\dots, m'_{d'})$, equivalently if it is torus-equivariant with respect to a map of tori $\gg_m^d\to \gg_m^{d'}$. For example the set of maps of log structures between $\pt_{log}^d$ and $\pt_{log}^{d'}$ is in bijection with $d\times d'$ matrices with positive integer coefficients (here matrices with integer coefficients classify torus maps $\gg_m^d\to \gg_m^{d'}$ and the positivity condition ensures that they extend to the partial compactification $\aa^d\to \aa^{d'}$). 
\item Suppose $X$ is a scheme with $D$ a strict normal-crossings divisor and $\underline{L} = (L_1,\dots, L_d)$ a tuple of line bundles. Suppose that $Y : = \bar{D}^k_\alpha$ is a closed normal-crossings stratification component of dimension $k$ and codimension $c = n-k$. Let $$\iota:Y\to X$$ be the closed embedding of this stratum. Then then there are $c$ distinct codimension-one closed strata $\bar{D}_1^{n-1},\dots, \bar{D}_c^{n-1}$ which contain $Y,$ and we have a canonical identity relating normal bundles $$N_XY \cong \bigoplus_{i=1}^c N_i,$$ for $$N_i : = \iota_Y^*N_X\bar{D}_i^{n-1}.$$ Let $D_Y$ be the union of all normal-crossings strata contained in $Y$ of higher codimension. Then we have the following canonical isomorphism: $$\iota_Y^*\big((X, D, \underline{L})_{log}\big)\cong \big(Y, D_Y, (L_1,L_2,\dots, L_d, N_1,N_2,\dots, N_d)\big).$$
Note that this immediately permits us to classify maps between strict normal-crossings schemes, so long as they lie over the embedding of a normal-crossing stratum. This will be the case for the maps defining our operad structure on $\fldlog.$ 
\end{enumerate}
\subsection{The log operad of framed little curves}\label{sec:fld_def}
We can now define the log analogue of the operad of framed little disks in terms of the minimalistic sketch of log geometry of the previous section. Note that our definition of $\fldlog$ here is equivalent to the less explicit moduli-theoretic definition of the reduced genus zero ``framed formal curves'' operad $$\mathfrak{ffc}_{g=0}$$ in \cite{ffc}. Recall that an operad $O$ in a category $\C$ with symmetric monoidal structure $\times$ is a collection of ``space of operations'' objects $O_n\in \C$ together with composition maps $\comp_i:O_n\times O_m\to O_{m+n-1}$ for $1\le i\le n,$ and symmetric group actions $\Sigma_n\curvearrowright O_n,$ satisfying certain compatibilities. In the presence of a covariant functor of points $S:\C\to \t{Sets},$ the sets $S(O_n)$ represent $n\to 1$ operations in an algebra structure, and the composition map $\comp_i:(o_m, o_n)\mapsto o_n\circ_i o_m$ represents the operation obtained by plugging in the output of $o_m$ as the $i$th input of $o_n$ (keeping all inputs of $o_m$ and the $n-1$ remaining inputs of $o_n$ free), and the $\Sigma_n$ action produces new operations from old ones by permuting the inputs. 

\begin{rmk}[A note about identity operations]
In this text we will prove formality for non-unital operads, as the formality property for a unital operad follows from formality for the corresponding non-unital operad: see for example \cite{roig}. In particular our model $\fldlog$ will be non-unital (though it can be made unital, see \cite{ffc}). Nevertheless to get a more concrete unital quasiisomorphism, it is possible to show that after taking chains, all our quasiisomorphisms are compatible with appropriate unital structure.
\end{rmk}

In particular, in order to define a log operad $\fldlog$ in the category of log schemes, we need to define log spaces $\fldlog_n$ of $n$-ary operations, the composition maps $\comp_i: \fldlog_m\times \fldlog_n\to \fldlog_{m+n-1}$ and permutation actions $\Sigma_n\curvearrowright \fldlog_n.$ 

We begin by reminding the reader about a closely related operad in the category of (ordinary) algebraic varieties, which we call the Deligne-Mumford-Knudson (DMK) operad. Namely, define for $n\ge 2,$ $$\DMK_n : = \bar{\M}_{0,n+1},$$ the moduli space of stable genus zero nodal curves with $n+1$ marked points labeled $x_0, x_1,\dots, x_n$. We think of the point $x_0$ as representing the output, and all $x_{\ge 1}$ as inputs. There is an obvious action of the symmetric group $\Sigma_n$ permuting inputs and we define composition morphisms $$\comp_i:\bar{\M}_{0,m+1}\times \bar{\M}_{0,n+1}$$ as maps of moduli representing the geometric gluing construction \begin{align*}(X, x_0, \dots, x_{m+1}) \comp_i (X', x_0',\dots, x_{n+1}) : &\cong \\ \left(\frac{X\sqcup X'}{(x_0'\sim x_i)}, x_0, x_1\dots, x_{i-1}, x_1', \dots, x_{m}', x_{i+1}, \dots, x_{n}\right)&
\end{align*}
(here as stable genus zero curves have no automorphisms, the isomorphism assignment uniquely defines the map). 

Checking these conditions on $\DMK_{\ge 2}$ satisfy the operad axioms is equivalent to checking that (up to isomorphism), the composite glueing $\frac{X\sqcup Y\sqcup Z}{x\sim y, y'\sim z}$ along disjoint pairs of points is independent of the order of spaces glued, a tautology. We formally extend the operad structure on $\DMK_{\ge 2}$ above to a unital operad by defining $$\DMK_1 : = \pt.$$ It is helpful to think of the point in $\DMK_1$ as classifying a ``fully collapsed'' genus $0$ curve with a single ``node'' point and no $1$-dimensional components, and with marked input and output (which happen to coincide in this case). 

Now we define $\fldlog$ to be a log operad on top of the operad $\DMK.$ Namely, let $D_n\subset \DMK_n$ be the normal-crossings divisor classifying all strictly nodal curves in $\bar{\M}_{0, n+1}.$ Let $$D_{m,m'}: = \comp_1(\DMK_m\times \DMK_m')\subset \DMK_n$$ be the image of the first composition map from $\DMK_m\times \DMK_{m'}$ for $m+m' = n+1.$ Let $L_k = T_k^*$ (for $k = 0,\dots, n$) be the line bundle over $\DMK$ classifying the cotangent line to the marked point $x_k$.
We have the following crucial lemma.
\begin{lm}
\begin{enumerate} 
\item\label{jio1} $\comp_1:\DMK_m\times \DMK_{m'}\to \DMK_n$ is an embedding of a closed normal-crossings component.
\item\label{jio2} $D = \cup \sigma(D_{m,m'})$ for $\sigma\in \Sigma_n$ ranging over $(m', m-1)$ shuffles. 
\item\label{jio3} $\comp_i^* N_{\DMK_n}(D)$ is canonically isomorphic to the tensor product line bundle $L_i^*\boxtimes L_0^*$ on $\DMK_m\times \DMK_{m'}.$
\end{enumerate}
\end{lm}
\begin{proof}
Parts \ref{jio1} and \ref{jio2} are equivalent to the well-known fact that the normal-crossings boundary $\bar{\M}_{0,n}\setminus \M_{0,n}$ consists of a union of moduli spaces classifying genus-zero curves that have a node that splits the set of marked points via a fixed bipartite partition. See for example Section 2.5.1 of \cite{cavalieri} for \ref{jio3}. 
\end{proof}

Now we are ready to define the operad $\fldlog.$ 
\begin{defi}
Define $$\fldlog_n: = (\DMK_n, D_n, \underline{L})$$ for $\underline{L} = (L_0, L_1,\dots, L_n)$ the tuple of cotangent bundles at all the marked points. Define $\Sigma_n$-action by permuting the marked input points $x_1,\dots, x_n$ and the corresponding line bundles $L_1, \dots, L_n.$ Define gluing maps $\comp_i^{\fldlog}: \fldlog_m\times \fldlog_n\to \fldlog_{m+n-1}$ as the composition 
\[
\begin{tikzcd} \fldlog_m\times \fldlog_n\arrow[r, "\alpha_{m,n}"]& (\comp_i^{\DMK})^*\DMK_{m+n-1}\arrow[r, "\iota_{m,n}"]& \DMK_{m+n-1},
\end{tikzcd}
\]
where $\iota_{m,n}$ above is the universal map from the pullback log structure, and $\alpha$ is a map between log structures to be defined below. By the standard characterization of the normal bundle to the boundary of Deligne-Mumford space we have
\begin{align*}(\comp_i^{\DMK})^*\DMK_{m+n-1}&\cong \\ \left(\DMK_m\times \DMK_n, D, \big(\comp_i^*(L_1), \dots, \comp_i^*(L_{m+n-1}), N_{\DMK_{m+n-1}}\comp_i\big)_{log}\right).&\end{align*}
\end{defi}
Now observe that each $(\comp_i^{\DMK})^*(L_k)$ is a bundle on $\bar{\M}_{0,n+1}\times \bar{\M}_{0,m+1}$ which for each pair of curves $(X, X')$ classifies the tangent line at a marked point of the glued curve $\frac{X\sqcup X'}{x_i'\sim x_0},$ equivalently the tangent line at either some $x_j$ or $x'_j,$ depending on $k$. The explicit matching doesn't matter very much, but explicitly we have \[\comp_i^*(L_k)\cong \begin{cases} 
L_k\boxtimes \oo_{\bar{\M}0,m+1}, & k<i\\
\oo_{\bar{\M}n+1, 0}\boxtimes L_{k-i+1}, & i\le k\le i+m-1\\
L_{k-m+1}, & i\ge i+m.
\end{cases}
\]
Observe also that the pulled back normal line bundle to the ``$i$-composition divisor'' $N_{\DMK_{m+n-1}}\comp_i$ on $\bar{\M}_{0,n+1}\times \bar{\M}_{0,m+1}$ is $T_i\boxtimes T_0\cong L_i\boxtimes L_0,$ the tensor product of the tangent bundles at the two glued points. Therefore if $j\neq i$ indexes a marked points of the curve $X$ and $k\neq 0$ a marked point of $X',$ we can define the (homogeneous) map $\alpha$ from each $L_i\boxtimes \oo$ to $\comp_i^*(\underline{L}_{0, m+n-1})$ by sending the corresponding summand isomorphically to the corresponding tangent line of the glued curve. It remains to define a map on the total space of $L_i\boxtimes \oo \oplus \oo\boxtimes L_0$ to the pulled back normal bundle $N_{\DMK_{m+n-1}}\comp_i\cong L_i\boxtimes L_0.$ We define this to be the quadratic multiplication map from the two-dimensional affine bundle $\aa L_i\boxtimes \oo \oplus \oo\boxtimes L_0$ to the tensor product of the two coordinates $L_i\boxtimes L_0.$ 

Note that we assumed that $n\ge 2$ in the above. We define $$\fldlog_1 : = \pt_{log}.$$ (The log point, equivalently $(\pt, k)_{log}.$) We define the $1\to 1$ operad compositions structure (equivalently, monoid structure) on $\fldlog_1$ as that induced from multiplicative monoid structure on $k$ (this is the standard monoidal structure on $\pt_{log}$). To define action maps $\comp_i:\fldlog_n\times \pt_\log\to \fldlog_n$ and $\comp_0:\pt_\log\times \fldlog_n\to \fldlog_n$ we need to specity an equivariant action of $(\aa^1,\cdot)$ on the affine space $L_0\times\dots \times L_n;$ we do this by having $\aa^1$ act linearly on the corresponding factor $L_i$ (with $i \in \{1,2,\dots, n\}$ for right action and $i = 0$ for left action of $\fldlog_n$). 

To check that this defines an operad structure, one must verify several standard relations, chief among them the associativity relation on a pair of operations of type $\comp_i.$ This follows from the fact that the normal bundle of a codimension-two component of $\bar{\M_{0,n}}$ is the direct sum of the two one-dimensional components that intersect in it; alternatively, it can be deduced by observing that our explicit construction coincides with the moduli-theoretic operad composition operations in \cite{ffc}. 

\subsection{Forgetful maps and $\fldlog_1$}\label{sec:forgetful} If $X$ is a stable curve with marked points indexed by a finite set $\Gamma$ and $\Gamma'\subset \Gamma$ is a subset of order $\ge 3$, the curve $X_{\Gamma'}$ is defined as the curve obtained from $X$ by forgetting all points not in $\Gamma'$ and contracting all unstable components. This induces ``forgetful'' maps $\forg_{\Gamma, \Gamma'}:\bar{\M}_{0,n}\to \bar{\M}_{0,n'}$ for $n\ge n'\ge 3.$ This map is compatible with normal-crossings structure, and it is straightforward to lift this map to a canonical map $\fldlog_{n-1}\to \fldlog_{n'-1}$ for $n\ge n'\ge 3.$ Though we will not use this until section \ref{sec:ld}, it will be convenient for us to extend this to $n'\ge 2,$ i.e., define maps $$\forg_{[n], \{i\}} :\fldlog_n\to \fldlog_1 (=\pt_{log})$$ (which later we will call $\theta^{log}_i$, in analogy with certain geometric maps) corresponding to ``forgetting all inputs except the $i$th input''. A convenient way to do so is as follows: let $j:\bar{\M}_{0,n+1}\to \bar{\M}_{0, 2(n+1)}$ be the map that takes a curve and glues a curve with three labels to each point. Then (by comparing normal bundles) we see that $\fldlog_n$ for $n\ge 2$ canonically fits into the following pullback diagram, with $\M_{0,n+1}^{log}$ the space $(\bar{\M}_{0,n+1},D_{0,n+1})_{log}$:
\[
\begin{tikzcd}
\fldlog_n \arrow[r, "j^{log}"]\arrow[d, "\pi"] & \M_{0,2(n+1)}^{log}\arrow[d, "\pi"]\\
\bar{\M}_{0,n+1}\arrow[r, "j"] & \bar{\M}_{0,2(n+1)}.
\end{tikzcd}
\]
Here in the moduli problem of $\bar{\M}_{0,2(n+1)}$ one should label the marked points $x_0,\dots, x_n, y_0,\dots, y_n$ and the image of $\bar{\M}_{0,n+1}$ under $j$ is the set of curves where each pair $x_i, y_i$ are on their own nodal component with only one node. 

Now the forgetful maps $\forg: \fldlog_n\to \fldlog_{n'}$ are the induced maps on pullback log schemes for the following map of triples of forgetful maps
\[
\begin{tikzcd}
&\M_{0,2(n+1)}^{log}\arrow[d, "\pi"]\arrow[dr]&\\
\bar{\M}_{0, n+1}\arrow[r, "j"]\arrow[dr]&\bar{\M}_{0,2(n+1)}\arrow[dr] & \M_{0,2(n'+1)}^{log}\arrow[d, "\pi"]\\
&\bar{\M}_{0,n'+1}\arrow[r, "j"] & \bar{\M}_{0,2(n'+1)}.
\end{tikzcd}
\]
We note that the space $\fldlog_1\cong \pt_{log}$ fits into the diagram 
\[
\begin{tikzcd}
\fldlog_1 \arrow[r, "j"]\arrow[d, "\pi"] & \M_{0,4}^{log}\arrow[d, "\pi"]\\
\pt\arrow[r, "j"] & \bar{\M}_{0,4},
\end{tikzcd}
\]
Here we label the points of $\M_{0,4}$ by $x_i, y_i$ for $0\le i\le 1,$ and the point $j(\pt)$ is the unique nodal curve with $x_0, y_0$ on one nodal component and $x_1, y_1$ on the other. We can now define forgetful maps \[\theta^{log}_i:\fldlog_n\to \fldlog_1\] as the induced maps on pullbacks for the diagram 
\[
\begin{tikzcd}
&\M_{0,2(n+1)}^{log}\arrow[d, "\pi"]\arrow[dr]&\\
\bar{\M}_{0, n+1}\arrow[r, "j"]\arrow[dr]&\bar{\M}_{0,2(n+1)}\arrow[dr] & \M_{0,4}^{log}\arrow[d, "\pi"]\\
&\pt\arrow[r, "j"] & \bar{\M}_{0,4},
\end{tikzcd}
\]
where the forgetful map $\bar{\M}_{0,2(n+1)}\to \bar{\M}_{0,4}$ forgets all indices $x_k, y_k$ for $k$ not in $\{0,i\}.$ Note that this gives additional meaning to the notion that one should consider $\bar{\M}_{0,2}$ to be a single point, classifying the ``single node'' curve (with dual graph the ``open edge'' graph). 
\begin{rmk} In fact, the structure of $\fldlog_n$ together with all forgetful maps is equivalent to an extension of the operad $\fldlog$ to an operad $\fldlog^+$ with $0\to 1$ operations classified by $\pt;$ ``forgetting'' an index is then the operation of ``gluing'' the point of $\fldlog^+_0$ at that index. We will not use this fact here.
\end{rmk}

\section{De Rham cohomology and Hodge theory}\label{sec:hodge_thy}
We review here the topological and analytic aspects of logarithmic geometry over $\cc,$ focusing on de Rham cohomology and Hodge-de Rham comparison. We take in this section $k = \cc$ and study the category $\t{LogSch}_k$ of $\cc$-log schemes.
\subsection{Kato-Nakayama spaces, the de Rham comparison and acyclicity}\label{sec:KN}
There is a category $\t{AnLog}$ of complex analytic spaces with log structure, and an \emph{analytic realization} functor $$\X\mapsto \X^{an}$$ from log schemes over $\cc$ to log analytic schemes, extending the functor $X\mapsto X^{an}$ from schemes with trivial log structure to analytic schemes with trivial log structure. Log schemes also have a wonderfully behaved \emph{topological} realization (discovered by Kato and Nakayama), but unlike the case of ordinary schemes, this is not equivalent to the analytic realization, i.e., the ``underlying topological space'' of a log scheme is not provided by the $\cc$-points with analytic topology: rather, this has to be modified to take into account the logarithmic structure. 

The ``corrected'' topological realization functor $$\X\mapsto \X^{top}:\t{LogSch}\to \t{Top}$$ is defined\footnote{What we denote $\X^{top}$ is written $\X^{log}$ in the original paper \cite{kato-nakayama}; we use the superscript ``top'' for topological to avoid confusion with log analogues of classical objects like $\M_{g,n}^{log}$} as 
$$\X^{top} := \X(\pt_{KN})$$ with $\pt_{KN}$ the ``Kato-Nakayama point'', a certain nontrivial log structure on $\spec(\cc)$. (Note that this structure is very different from the ``log point'' structure, $\pt_{log},$ which is defined over any base field.) The Kato-Nakayama point carries a certain analytic structure which gives $\X^{top}$ the structure of, first, a topological space and second, a locally ringed space. When $\X = X$ has trivial log structure, we have naturally $$\X^{top}\cong \X^{an}$$ as locally ringed topological spaces, but for general log schemes $\X^{an}$ might not be analytic, and indeed might have odd (real) dimension. For a scheme of type $(X, D)_{log},$ its analytification is the real blow-up $$\t{Bl}_D(X(\cc))$$ of the divisor $D(\cc)\subset X(\cc).$ This is a real $2n$-dimensional manifold with corners (for $n = \t{dim}(X)$) whose interior is canonically an $n$-dimensional complex analytic manifold, namely $(X\setminus D)(\cc).$ Note that in this case the map of topological spaces $U^{top}\to \X^{top}$ is a homotopy equivalence, for $U = X\setminus D$ the locus of trivial log structure on $\X$. The topological realization of the log point is the unit circle, $$(\pt_{log})^{top}\cong S^1.$$ 

Before proceeding, we give the following lemma, proved in \cite{ffc}, establishing a comparison between the Kato-Nakayama realization of $\fldlog$ and framed little disks.
\begin{lm}
\begin{enumerate}
\item\label{trfd:1} The Kato-Nakayama topological realization $\fldlog^{top}$ is canonically isomorphic as an operad to the operad $\underline{\mathcal{N}}$ introduced in the paper \cite{ksv} by Kimura, Stasheff and Voronov. 
\item\label{trfd:2} The topological operad $\fldlog^{top}$ is related to $FLD$ by a pair of explicit homotopy equivalences of operads.
\end{enumerate}
\end{lm}
\begin{proof}
Part (\ref{trfd:1}) is proven in \cite{ffc}. In brief: the Kimura-Stasheff-Voronov operad has space of operations $\underline{\mathcal{N}}_n$ defined as an $(S^1)^{n+1}$-bundle over the real blow-up of the boundary in $\bar{\M}_{0,n+1}$. On the other hand, the space $\fldlog$ is a $\pt_{log}^{n+1}$-bundle over the normal-crossings log space $(\bar{\M}_{0, n+1}, D)^{log}$, and so its geometric realization is also an $(S^1)^n$-bundle over the real blow-up of the boundary of $\bar{\M}_{0,n+1}$. Essentially by definition, the two bundles are the same. It is then an exercise done in \cite{ffc} to check that the operad composition maps match up. 

Part (\ref{trfd:2}) is proven in \cite{ksv}; see also 2.2 in \cite{gs_formality}. 
\end{proof}

Now there is a theory of GAGA (G\'eom\'etrie Alg\'ebrique-G\'eom\'etrie Analytique, after Serre \cite{gaga}) comparisons of invariants for a log scheme $\X$ which is richer than for non-logarithmic schemes, as it involves comparisons of invariants for not two but three objects $\X$, $\X^{an}$ and $\X^{top}$ with their accompanying local ringed and logarithmic structures. The relevant comparison in our case is the pair of isomorphisms \[H^*_{dR}(\X)\cong H^*_{dR}(\X^{an})\] and \[H^*_{dR}(\X^{an}) \cong H^*_{\sing}(\X^{top}, \cc)\] which hold for any $\Omega$-smooth log scheme $\X$ over $\cc.$ We will in fact use a refinement of these homology isomorphisms to chain level, which is easier to state for schemes with a nice acyclicity property. 
\begin{defi} We say that a log scheme $\X$ is \emph{acyclic} if the sheaves $\Omega^k(\X)$ are acyclic for all $k$, i.e., $H^{\ge 1}(\Omega^k(\X), X) = 0.$
\end{defi}
For example any log scheme with affine underlying scheme is acyclic (since coherent sheaves have no higher cohomology), meaning that any log scheme has an acyclic Zariski cover. The complex of de Rham chains is easiest to write down for $\X$ acyclic, where we simply write \[C^*_{dR, acyc}(\X) : = \left(H^0\Omega^*(\X), d\right).\] This definition can be extended to a DG functor on non-acyclic sheaves by gluing (it is canonical in an $\infty$-categorical sense). We will only need the acyclic statement here.

\begin{thm}\label{bettitodr} 
Let $\mathrm{Acyc}$ be a small symmetric monoidal model for the the category of acyclic log schemes (recall that our log schemes are coherent fs schemes of finite type, so isomorphism classes form a set). Then there is (canonically in an $\infty$-categorical sense) a sequence of symmetric monoidal quasiisomorphisms between (lax, contravariant) symmetric monoidal functors $\mathrm{Acyc}\to \mathrm{dg-Vect}$ between the functor $\X\mapsto C^*_{\sing}(\X^{top})$ of singular cochains and $C^*_{dR, acyc}$ of log de Rham cochains. 
\end{thm}
\begin{proof}
This follows from a simple $\infty$-categorical extension of Kato and Nakayama's Betti-to-de Rham comparison result, see the appendix \ref{appendix}. 
\end{proof}

\begin{rmk}
Note that this theorem implies a Hodge-to-de Rham equivalence for the $\infty$-category of all (fs) log schemes, by viewing fs log schemes as functors of points on affine fs log schemes, which are automatically acyclic. 
\end{rmk}

\subsection{Hodge to de Rham and proper acyclicity}\label{sec:pa}
Like in the classical case of smooth schemes, the hypercohomology interpretation of $H^*_{dR}$ implies a Hodge to de Rham spectral sequence with $E1$ term $$H^p (X,\Omega^q_{\X/X})\implies H^*_{dR}(\X).$$ Also analogously to the classical case, this spectral sequence degenerates when $\X$ is smooth and proper \cite{kato} (see also \cite{hablicsek}).\footnote{In fact this degeneration also holds when $\X$ is $\Omega$-smooth with underlying scheme $X$ proper. As the author could not find a reference for this more general fact in the literature, we will deduce a statement of this type by converting an $\Omega$-smooth scheme into a log smooth scheme with the same Hodge structure (something that is always possible using toric geometry).} This spectral sequence is associated to a chain-level \emph{Hodge} filtration on $C^*(\X)$. In the previous section we defined a log scheme to be acyclic if $H^{\ge 1}\Omega^k_\X = 0$ for all $k$. The main new input into our formality splitting is the acyclicity property. Note that an $n$-dimensional smooth and proper variety with trivial log structure has a dualizing class in $H^n\Omega^n_X,$ and hence cannot be acyclic unless $n=0.$ However a nontrivial log variety can be both acyclic and proper (we say that a log scheme is proper if its underlying scheme is proper). A key example is the smooth and proper log scheme $$\X = (\pp^1, D)_{log}$$ for $D$ a divisor with $d \ge 1$ points. Here acyclicity is elementary: indeed, the sheaf of functions $\Omega^0$ does not depend on log structure giving $\Omega^0_{\X/X} = \oo_{\pp^1},$ which is acyclic, and $\Omega^1_{\X/X} = \Omega^1_{\pp^1}(D) \cong \oo(d-2)$ is also acyclic. Note that log smoothness is not required for proper acyclicity: for example, the log point $\pt_{log}$ is trivially proper acyclic. 

Let $$\t{PALog}$$ be the category of proper acyclic $\Omega$-smooth log schemes (a full subcategory of $\t{LogSch}$). Hodge to de Rham degeneration implies the following theorem.

\begin{thm}
Over a field $k$ of characteristic $0$, there is a canonical formality quasiisomorphism natural transformation $i:H^*_{dR}\to C^*_{dR}$ between the two DG functors $\t{PALog}\to \t{dgVect}$ (for $\t{dgVect}$ the category of complexes). Moreover, $i$ is compatible with (lax) symmetric monoidal structure.
\end{thm}
\begin{proof} By our definition, $C^*_{dR}$ is defined for acyclic sheaves as the complex $(H^0\Omega^k_\X, d_{dR})$. Hodge-to-de Rham degeneration for proper schemes implies that $d_{dR} = 0$ for each $\X\in \t{PALog}.$ Thus the functor $C^*_{dR}$ as a (lax) symmetric monoidal functor from $\t{PALog}$ to the category of complexes factors through the subcategory of complexes with zero differential (and is in fact strict symmetric monoidal).
\end{proof}
Functoriality and symmetric monoidicity of this natural transformation implies that any ``algebraic co-structure'' on de Rham cochains associated to an algebraic structure on proper acyclic spaces is itself canonically proper. In particular we have the following corollary.

\begin{cor}[Theorem \ref{thm:PALog-main} of Section \ref{sec:results}]\label{thm:padr-formality}
The Betti cochain functor $C^*_{\t{Betti}}:\X\mapsto C^*_{sing}(\X^{top},\cc)$ is \emph{formal} as a symmetric monoidal functor of $\infty$-categories when restricted to $\t{PALog}\subset \t{LogSch}$, i.e., it is canonically quasiisomorphic to the Betti cohomology functor $\X\mapsto H^*(\X^{top}, \cc).$
\end{cor}
\begin{proof}
We combine Theorem \ref{thm:padr-formality} with Theorem \ref{bettitodr}.
\end{proof}
\begin{cor}\label{cor:ops}
If $X\in \t{PALog}$ is a pure acyclic log space then the cohomology algebra $C^*(X,\cc)$ is canonically formal, and if $\mathbb{O}$ is an operad valued in $\t{PALog}$ then $C^*(\mathbb{O}, \cc)$ is canonically a formal co-operad.
\end{cor}
\begin{proof}
This follows from the previous corollary, applied either to the comonad $X$ (viewed as a commutative coalgebra in $\t{PALog}$ in the standard way, with coproduct $\Delta:X\to X\times X$) or to the operad $O$. Note that the two can be combined to obtain formality of $C^*(\mathbb{O},\cc)$ in the $\infty$-category of cooperads of commutative algebras.
\end{proof} 
Armed with Corollary \ref{cor:ops}, it remains only to show that the operad $\fldlog$ is an operad of spaces in $\t{PALog}$ in order to get an explicit formality quasiisomorphism for the corresponding topological operad $FLD$. 
\subsection{Proper acyclicity of $\fldlog$}\label{sec:operad_pa}
In order to prove formality of framed little disks it remains to prove the following theorem. 
\begin{thm}\label{operad_pa_property}
The spaces $\fldlog_n$ are proper acyclic.
\end{thm}
\begin{proof}
As the underlying scheme of $\fldlog_n$ is $\bar{\M}_{0, n+1}$ (or $\pt$ for $n=1$), properness is automatic. It remains to demonstrate acyclicity of $\Omega^k_{log}$. We give a proof using Hodge theory, using interactions between the Hodge and weight filtrations (we note that more explicit proofs of this computation in the cohomology of coherent sheaves are possible. See also section 3 of \cite{getzler} for a related computation). We say that a (smooth but not necessarily proper) algebraic variety $X$ (with trivial log structure) is \emph{$2$-pure} if the associated graded $Gr_W^k(H^*(X, \cc))$ is purely in cohomological degree $\frac{k}{2}$ (and in particular, is trivial for $k$ odd). The quintessential example is $X = \gg_m,$ which has $H^1(\gg_m) = \cc(1)$ the Tate Hodge module, pure of weight $2$. In fact all $2$-pure varieties can be seen to be of Tate type.\footnote{This definition can be generalized to a notion of $\alpha$-purity for any rational number $\alpha$. In particular $1$-pure schemes are schemes with $Gr_W^k(H^*)$ purely in cohomological degree $k$, and include all smooth proper varieties. See also \cite{ch_mixed}.}

We begin by proving that a \emph{smooth} proper log variety that compactifies a $2$-pure variety is proper acyclic. 

\begin{lm}\label{2pureacyc}
Suppose that $U$ is a $2$-pure smooth algebraic variety over $\cc$ and $X$ is a normal-crossings compactification of $U$ with $D = X\setminus U$. Then $$\X : = (X, D)_{log}$$ is acyclic.
\end{lm}
\begin{rmk} In fact, it is possible to extend this statement to $\Omega$-smooth log varieties, though we do not need the full generality here. Namely, there is a way to associate motives more generally to log varieties (explained for example in upcoming work of Vologodsky et al, \cite{vologodsky}), and it is possible to show that an $\Omega$-smooth log variety is proper acyclic if and only it is proper and its motive is $2$-pure. 
\end{rmk}
\begin{proof}
There is a weight spectral sequence for $H^*(\X)$ built out of the cohomology groups of $X$ and the closed strata of $D$, with $E1$ term as follows:
\[
\begin{tikzcd}
\ddots& \vdots           & \vdots             & \vdots   \\
\ldots& H^0(D^2) \arrow[r, "d_1"] & H^2(D^1) \arrow[r, "d_1" ]& H^4(X)    \\
\ldots& 0                 & H^1(D^1) \arrow[r, "d_1"] & H^3(X)       \\
\ldots& 0                 & H^0(D^1) \arrow[r, "d_1"]  & H^2(X)         \\
\ldots& 0                 & 0                  & H^1(X)          \\
\ldots& 0                 & 0                  & H^0(X)        ,
\end{tikzcd}
\]
for $D^k : = \sqcup \tilde{D}^{n-k}_i$ is the disjoint union of normalizations of irreducible components of closed codimension-$k$ strata of $D$, and with differential $d_1$ a Gisin differential. The weight filtration splits on the $E1$ term, with $E_1^{p,q} = H^{p-2q}(D^p)$ in weight $p$; so $2$-purity is equivalent to vanishing of all terms except on the $E\infty$ page except $E_\infty^{2p, p}$ (note that the spectral sequence degenerates at the $E_2$ page, hence the vanishing is also true on $E_2$). The weight filtration is compatible with the Hodge filtration on columns, with $F_{Hodge}^p(D^q)$ concentrated in $H^{\ge p}.$ Thus $2$-purity implies that $F_{Hodge}^{\ge 1} = 0,$ i.e., the Hodge filtration is concentrated in degree $0$. Now the associated graded $Gr_{Hodge}^p(U)$ is computed by the complex $(H^p(\Omega^*_U), d_{dR}),$ which can be computed on in terms of log forms in a normal-crossings compactification. Hodge-to-de Rham degneration then guarantees that $H^{\ge 1}\Omega^*(X,D)_{log} = 0.$ 
\end{proof}
The advantage of the weight filtration over the Hodge filtration that is relevant to our case is its good behavior in families. In particular the weight filtration behaves ``flatly'' in families, and we have the following lemma.
\begin{lm}\label{purity_descent}
Suppose $\pi:E\to B$ is a smooth family of schemes over a smooth base $B$ (over $\cc$). Suppose that $B$ is $2$-pure and a fiber $F\subset E$ over some point of $B$ is $2$-pure. Then $E$ is $2$-pure.
\end{lm}
\begin{proof}
This follows from appropriate compatibility of weight filtrations with the Serre spectral sequence.
\end{proof}
\begin{cor}\label{purity_moduli}
The spaces $\M_{0,n}$ are $2$-pure.
\end{cor}
\begin{proof}
Indeed, $\M_{0,3} = \pt$ is pure acyclic and there is a smooth fibration $\M_{0,n+1}\to \M_{0,n}$ with fibers isomorphic to $n$ times punctured $\pp^1$ (which is $2$-pure as $H^1(\pp^1\setminus D) \cong \cc^{d-1}(1),$ for $D$ a collection of $d\ge 3$ distinct points), whose $H^1$ motive is isomorphic to $\cc(1)^{d-1}$, of weight $2$. \end{proof}
This implies that the moduli space of marked genus zero log curves $(\bar{\M}_{0,n},D)_{log}$ is proper acyclic. To deduce a result for $\fldlog_{n} = (\bar{\M}_{0,n+1},D, \underline{L})_{log}$ we need to deal with the additional line bundles $L_1,\dots, L_n.$ Define the smooth log variety $$\pp(\bar{\M}_{0,n+1}, \underline{L})$$ to be the (total space of the) $(\pp^1)^{n+1}$-bundle $\pp(L_1)\times \dots\times \pp(L_n)$ which compactifies the total space of $L_0\oplus \dots \oplus L_n$. We introduce a log structure $$\pp(\bar{\M}_{0,n+1}, \underline{L})_{log} = \left(\pp(\bar{\M}_{0,n+1},\underline{L})_{log}, \pp D\right),$$ where $\pp D$ is the union of the preimage of $D\subset \bar{\M}_{0,n+1}$ and the union of the $0$-section and the $\infty$-section of each $\pp^1$-bundle. We have $$\iota: (\bar{\M}_{0,n+1}, D, \underline{L})_{log}\hookrightarrow \pp(\bar{\M}_{0,n+1}, \underline{L})_{log}$$ embedded as the induced log structure on the simultaneous zero section of all $\pp^1$-bundles. Now note that the embedding $\pt_{log}\hookrightarrow (\pp^1, 0\sqcup 1)_{log}$ embedded as the log structure at $0$ induces an isomorphism on global sections $\Omega^*$ (with both spaces proper acyclic). Taking a tensor power, we deduce the same for the embedding $\pt_{log}^{n+1}\subset (\pp^1, 0\sqcup \infty)_{log}^{n+1},$ hence arguing fiberwise we see that the map $\iota$ above induces an isomorphism on each $H^p\Omega^q$, and in particular preserves the acyclicity property. Thus it remains to prove that $\pp(\bar{\M}_{0,n+1}, \underline{L})_{log}$ is pure acyclic. Since the open scheme $\pp(\bar{\M}_{0,n+1}, \underline{L})\setminus \pp D$ is a $\gg_1^{n+1}$-bundle over $\M_{g,n}$ we see that it is $2$-pure by applying Lemma \ref{purity_descent} to \ref{purity_moduli}, hence done with Theorem \ref{operad_pa_property}. 
\end{proof}

\section{From framed little disks to little disks}\label{sec:ld}

The formality problem was first posed and first proven for the $E_2$ operad, equivalent to the operad $\LD$ of ordinary (i.e., not framed) little disks. The operads $\LD$ and $\FLD$ are closely related: indeed, $\FLD$ is a semidirect product of $\LD$ and $S^1.$ One aspect of this relationship is a fiber square of operads of topological spaces (defined below), as follows. 
\begin{equation}\label{eq:fld-square}
\begin{tikzcd}
\LD \arrow[r] \arrow[d] & \FLD\arrow[d, "\theta"]\\
\t{Comm} \arrow[r, "i"]& \t{Comm}^{S^1}.
\end{tikzcd}
\end{equation}
Here $\t{Comm}$ is the \emph{commutativity operad} (the terminal object in the category of operads), uniquely defined by $\t{Comm}_n = \pt$ for all $n$. It is so named because the category of algebras over this operad (viewed as a unital operad) is the category of commutative monoids. For any (not necessariy commutative or unital) topological monoid $G$, there is a ``$G$-equivariant commutativity'' operad $\t{Comm}^G,$ with category of algebras equal to the category of $G$-equivariant monoids. This operad is defined by $\t{Comm}^G_n : = G^n,$ with composition \[(g_1,\dots, g_n)\circ_i (g'_1, \dots, g'_k) : = (g_1,\dots, g_{i-1}, g_i g'_1, g_i g'_2, \dots, g_i g'_k, g_{i+1},\dots, g_n).\] The map $i:\t{Comm} \to \t{Comm}^{S^1}$ is the map on operads induced by the unique map of groups $1:\{e\}\to S^1$. Explicitly, it is defined by $i_n(*) = (1,1,\dots, 1)\in \t{Comm}^{S^1}_n.$ 

Recall that $\FLD_n$ classifies the data of a collection of $n$ nonoverlapping maps $\iota_1,\dots, \iota_n$ from closed disks $\dd^2$ to a single $\dd^2$ which are complex homotheties (i.e., compositions of a translation, scaling and rotation). Such a collection is determined by a collection of triples $(z_i, r_i, \theta_i)$ for $z_i\in D^1\subset \cc$ the center of the $i$th disk, $r_i\in \rr$ its radius and $\theta_i\in S^1$ its angle of rotation of the complex homothety $\iota_i$. 

Let $\theta: \FLD_n\to \t{Comm}^{S^1}_n$ be the map recording the angles, $(\theta_1,\dots, \theta_n)\in (S^1)^n.$ This is a map of operads. This defines the three lower right objects and maps of the square (\ref{eq:fld-square}). Now the operad $\LD\subset \FLD$ has spaces of operations $\LD_n\subset \FLD_n$ given by tuples of nonoverlapping \emph{real} homotheties $\iota_n:\dd^2\to \dd^2$, equivalently, elements of $\FLD$ with ``framing'' angles $\theta_1,\dots, \theta_n = 0.$ For each $n$ we thus have $\LD_n = \theta^{-1}(1,1,\dots, 1)\subset \FLD_n,$ verifying the commutativity and the pullback property of the diagram (\ref{eq:fld-square}).

Note that each map $\FLD_n\to (S^1)^n$ is a Hurewicz (therefore also a Serre) fibration, and so (since fibrancy for operads is inherited from spaces, see \cite{berger-moerdijk}), the diagram (\ref{eq:fld-square}) is a homotopy basechange diagram. We would like to get Hodge splitting properties for the diagram $\LD$ by re-interpreting this diagram in a logarithmic context, though there are some new complications here as we shall see. First we observe that the right column $\FLD\to \t{Comm}^{S^1}$ of (\ref{eq:fld-square}) has a logarithmic analog. Namely, recall (section \ref{sec:forgetful}) that for $i = 1,\dots, n,$ we have maps $\theta^{log}_i:\FLD_n\to \FLD_1.$ The space $\FLD_1 = \pt_{log}$ has the structure of a (non-unital) monoid in the category of log schemes, hence induces an ``equivariance'' operad and the maps 
\[\theta^{log} : = (\theta_1^{log},\dots, \theta_n^{log}): \fldlog_n\to \pt_{log}^n\] 
combine to a map of operads $\fldlog\to \t{Comm}^{\pt_{log}}$. A direct comparison (see Appendix) shows that after K-N realization, this map of operads is related (by a pair of quasiisomorphisms of maps of operads) to the map of topological operads $\theta:\fldlog\to \t{Comm}^{S^1}.$ We would like to draw a diagram of log operads analogous to the lower right three entries of (\ref{eq:fld-square}), and define an operad ``$\fldlog$'' as the pullback, but we run into a problem. Namely, $\pt_{log}$ is non-unital and indeed, there is no map $\pt\to \pt_{log}$, hence the bottom row of the diagram cannot be interpreted in the log category. In the paper \cite{ffc}, this problem is resolved by moving to a more flexible motivic category; however in this paper we use a more concrete solution, involving the de Rham complex of sheaves in the Kato-Nakayama realization. 

We begin by replacing this diagram with a quasiisomorphic one: namely, let $\rr(1)$ be the group isomorphic to $\rr,$ but understood as the group of purely imaginary complex numbers. The map $\exp:\rr\to S^1$ induces a map $\t{Comm}^\rr\to \t{Comm}^{S^1}.$ Define the operad $\tilde{\LD}$ to be the pullback of the diagram 
\begin{equation}\label{eq:ldtilde-square}
\begin{tikzcd}
\tilde{\LD} \arrow[r] \arrow[d] & \FLD\arrow[d, "\theta"]\\
\t{Comm}^\rr \arrow[r, "\iota"]& \t{Comm}^{S^1}.
\end{tikzcd}
\end{equation} 
(Geometrically, a point of $\tilde{\LD}_n$ classifies a point of $\FLD$ together with a homotopy class of paths in $S^1$ from each of the angles $\theta_i$ to $1\in S^1.$) The map $\pt\to \rr$ induces the homotopy equivalence of operads $\LD\to \tilde{\LD}.$  

Now after taking cochains, the last diagram can be compared to a diagram of logarithmic origin. Namely, recall that for $\X$ a logarithmic variety, $\X^{top}$ is a locally ringed space with ring of ``log holomorphic'' functions $\oo^{top}$. The sheaf $\oo^{top} = : \Omega^0_{top}$ is part of a complex \[(\Omega^*_{top}(\X), d) : = \Omega^0_{top}\to \Omega^1_{top}\to \dots \to \Omega^{n+d}_{top}\] (here $n+d$ is the ``log dimension''), which, assuming $\X$ is $\Omega$-smooth, resolves the sheaf $\cc^{top}$ of locally constant functions on $\X^{top}.$ For $\X = \pt_{log},$ let us choose a basepoint $\vec{1}\in \pt_{log}^{top}$ (the point of the exceptional fiber of the real blowup of $\cc$ in the direction of $1\in \cc$), and let $\tilde{\pt}_{log}^{top}$ be the universal cover with respect to this basepoint, canonically isomorphic to $\rr(1).$ Then the complex $\Omega^0(\pt_{log}^{top})\to \Omega^1(\pt_{log}^{top})$ has a lift to $\tilde{\pt}_{log}^{top},$ which we denote \[\Omega^0(\tilde{\pt}_{log}^{top}) \to \Omega^1(\tilde{\pt}_{log}^{top}).\]
Since this complex resolves $\cc_{\rr(1)},$ its global sections are one-dimensional, spanned by the unit global section $1\in \Omega^0(\tilde{\pt}_{log}^{top}),$ and there is no higher cohomology. 
\begin{rmk}
In fact, the complex of sheaves \[\Omega^0(\tilde{\pt}_{log}^{top}) \to \Omega^1(\tilde{\pt}_{log}^{top})\] is quite simple. $\Omega^0$ is the constant sheaf with fiber the polynomial algebra $\cc[\log]$ for $\log$ a variable (corresponding to the log function on the universal cover of $\cc^*$) and $\Omega^1 = d\log \cdot \cc[\log].$
\end{rmk}
Let $$\tilde{\LD}^{log, top}$$ be the covering of the Kato-Nakayama operad $\fldlog^{top}$ which fits into the pullback diagram 
\begin{equation}\label{eq:flc-square}
\begin{tikzcd}
\tilde{\LD}^{log, top} \arrow[r] \arrow[d] & \fldlog^{top}\arrow[d, "\theta^{log,top}"]\\
\t{Comm}^{\tilde{\pt}_{log}^{top}} \arrow[r, "\iota"]& \t{Comm}^{\pt_{log}^{top}}.
\end{tikzcd}
\end{equation}
Define the complex of sheaves $\Omega^*(\tilde{\LD}^{log, top}_n)$ to be the pullback of the complex of Kato-Nakayama log differential forms $\Omega^*(\fldlog^{top}_n)$ on the topological space $\fldlog^{top}_n$ to its cover. Note that this complex once again resolves the constant sheaf $\cc$ on the topological space $\tilde{\LD}^{log, top}_n.$ 
\begin{rmk}\label{rmk:fld_hodge}
The operad $\tilde{\LD}^{log,top}$ with the complex of logarithmic forms $\Omega^*(\tilde{\LD}^{log, top})$ is a natural ``analytic geometry'' home for the mixed Hodge structure on the operad of chains $C_*\LD$ given by Tamarkin's construction \cite{tamarkin_action} (see also \cite{ch_mixed}): in particular, both the Hodge and the weight filtration are clearly visible in this picture. 
Moreover, the complex of sheaves $\Omega^*(\tilde{\LD}^{log, top})$ has naturally a rational lattice (lifting the rational structure on $\Omega^*(\fldlog)$), giving naturally a de Rham lattice $C^*_{dR}(\LD).$ It can be checked that this is the same lattice as the one given in \cite{ffc}.
\end{rmk}

Applying the \v{C}ech cochains functor $C^* = : C^*_{\cech}$, we obtain the following diagram of cooperads in the category of cdga's:
\begin{equation}\label{eq:cdga-square}
\begin{tikzcd}
C^*(\tilde{\LD}^{log, top}, \Omega^*_{log, top}) & \arrow[l] C^*(\fldlog^{top}, \Omega^*_{log, top})\\
C^*(\t{Comm}^{\tilde{\pt}_{log}^{top}},\Omega^*_{log, top}) \arrow[u] & \arrow[l] C^*(\t{Comm}^{\pt_{log}^{top}},\Omega^*_{log, top})\arrow[u].
\end{tikzcd}
\end{equation}
Since each $\Omega^*_{log, top}$ is a resolution of a constant sheaf, and the diagram of spaces is a fibration, we see that this is a homotopy pushforward diagram in the category of cdga's for each space of operations (therefore also a homotopy pullback diagram in the category of co-operads of connective cdga's). Now consider the following diagram of triples of complexes. 
\begin{equation}\label{eq:triple-square}
\hspace*{-2.5cm}
\begin{tikzcd}
C^*(\t{Comm}^{\tilde{\pt}_{log}^{top}},\cc) \arrow[d]& \arrow[l] C^*(\t{Comm}^{\pt_{log}^{top}},\cc) \arrow[r]\arrow[d]& C^*(\fldlog^{top}, \cc)\arrow[d]\\ 
C^*(\t{Comm}^{\tilde{\pt}_{log}^{top}},\Omega^*_{log, top}) & \arrow[l] C^*(\t{Comm}^{\pt_{log}^{top}},\Omega^*_{log, top}) \arrow[r]& C^*(\fldlog^{top}, \Omega^*_{log, top})\\
\cc \arrow[u, "\iota"] & \arrow[l] H^*(\t{Comm}^{S^1},\cc)\arrow[r]\arrow[u, "\alpha"]& H^*(\FLD, \cc).\arrow[u, "\beta"]
\end{tikzcd}
\end{equation}
The maps between the first and second row are given by the map of complexes of sheaves $\cc\to \Omega^*_{log, top}$ which is a quasiisomorphism (by the de Rham resolution property stated above), the maps $\alpha, \beta$ between the bottom two rows are the formality maps of Theorem \ref{operad_pa_property} and $\iota$ is the embedding of constant functions. This is easily seen to be a commutative diagram with all vertical maps quasiisomorphisms. Now note that a triple of connective cdga's $R'\from R\to R''$ is ``pullback-exact'', i.e. satisfies the property that \[
\begin{tikzcd}
R\arrow[r]\arrow[d]&R'\arrow[d]\\
R''\arrow[r]&R'\otimes_R R''
\end{tikzcd}
\]
is a homotopy pushforward diagram, if either $R\to R'$ or $R\to R''$ is cofibrant in the Quillen model structure. For the top two rows, the map on the left is cofibrant (since it is induced by pullback of sections for a covering space), and for the bottom diagram, the map on the right is cofibrant since the map of (ordinary) topological operads $\FLD_n\to (S^1)^n$ is a fibration. Thus taking cdga pushout of each row gives a new triple of quasiisomorphisms of operads of cdga's, with the bottom pushout isomorphic to $H^*(\LD)$ and the top pushout mapping quasiisomorphically to $C^*(\tilde{\LD}^{log, top})$ which is related by a pair of quasiisomorphisms of topological operads to $C^*(\FLD).$ 

\section{Applications and extensions}

Here we sketch out very briefly several applications and extensions of the results of this paper and make some conjectures. Arguments in this section are sketches rather than full proofs.

\subsection{Integral splitting}
Our spaces $\fldlog_n$, the composition maps $\circ_i$ and the $\Sigma_n$ actions involved in the operad structure are all defined and $\Omega$-smooth over $\zz.$ Let $A^*_n$ be the complex \[A^*_n = \big(H^0\Omega^0_\zz(\fldlog_n)\to H^0\Omega^1_\zz(\fldlog_n)\to \dots\to H^0\Omega^{2n-1}(\fldlog_n)\big).\] Smoothness implies in particular that this is a complex of free $\zz$-modules which are submodules in the $\qq$-basechange; our formality result over $k = \qq$ then implies that the complexes $A^*_n$ are formal. By standard functoriality and symmetric monoidicity of $H^0$ (for connective complexes), we obtain a cooperad of graded spaces $A^*,$ which is an integral lattice in $H^*(\fldlog).$ The ``truncation'' maps $H^0(\Omega^k)\to C^*(\Omega^k)$ then define a map of co-operads of $\zz$-complexes \[A^*\to C^*_{dR, \zz}(\fldlog).\] Dualizing, we obtain a map of \emph{operads} \[C_*^{dR}(\fldlog, \zz): = \big(C^*_{dR}(\fldlog,\zz)\big)^\vee\to (A^*)^\vee;\] this implies that given any algebra over the formal $\zz$ \emph{operad} $(A^*)^\vee$ one obtains canonically an operad over the ``integral de Rham chains'' of $\fldlog,$ an integral version of de Rham splitting. It is natural to ask two questions.
\begin{question}
Is the ``formality map'' $A^*\to C^*_{dR,\zz}(\fldlog)$ a quasiisomorphism?
\end{question}
\begin{question}
Is the co-operad $A^*$ isomorphic (non-canonically) to the integral cohomology co-operad $H^*(\FLD)$, i.e., dual to the integral BV operad? 
\end{question}
The results of this paper imply that both of these are true after
extending coefficients to $\qq$, but they may well have torsion
obstructions integrally. Note that, as $C^*_{dR, \zz}(\fldlog)$
might not be quasiisomorphic to $C^*(\FLD, \zz)$ (though see the next
question), both of these may be true without implying that the
topological chains operad $C_*(\FLD, \zz)$ is formal.

\subsection{Prismatic cohomology} 
For a smooth and proper scheme $X$ over $\zz_p,$ there is a \emph{Prismatic cohomology} theory (\cite{bhatt-scholze}, \cite{bms}) with complex of cochains $C^*_{prism}(X, A_{inf})$ with coefficients in Fontaine's period ring $A_{inf}$ which interpolates at different points of $\t{Spec}(A_{inf})$ between characteristic-$p$ de Rham cohomology $C^*_{dR, \bar{\ff}_p}(X)$ and \'etale cohomology $C^*_\text{\'et}(X_{\qqbar_p}, \bar{\ff}_p)$ (equivalent to Betti cohomology with $\bar{\ff}_p$-coefficients). 
\begin{question}
Does there exist a Prismatic cohomology theory $C^*_{prism}(\fldlog_{\zz_p})$ over $A_{inf}$ which interpolates between \'etale and de Rham cohomology in an analogous way? 
\end{question}
This would follow from a sufficiently powerful \emph{logarithmic} $p$-adic Hodge theory for $\Omega$-smooth log varieties. Some steps towards such a theory are taken by T.\ Koshikawa and K. Cesnavicius in \cite{koshikawa} and \cite{ck_log}; see also \cite{log_coh}.

\subsection{Comparison with other rational de Rham theories}
The canonical splittings for $C_*(\LD,\cc)$ and $C_*(\FLD,\cc)$ constructed here are particularly interesting compared to previously known splittings, because they are the first explicitly constructed splittings which are compatible with a rational structure on $C_*(\LD,\cc)$; namely, the de Rham rational structure. This structure (or rather its dual, $C^*_{dR}$) was first constructed as part of a mixed Hodge structure in the paper \cite{ch_mixed}, where it was observed to follow from the Grothendieck-Teichm\"uller action discovered in \cite{tamarkin_action} (the corresponding action on the framed operad $\FLD$ follows from \cite{severa}). Another, log geometric, interpretation for this rational structure was given in \cite{ffc}. We give here without proof two other places where the rational lattice $C^*_{dR}(\LD, \qq)$ should appear in a canonical way.
\subsubsection{$\FLD$, homotopy pushout and moduli of nodal curves}
In the paper \cite{dc} (see also \cite{dv}), Drummond-Cole shows that the topological Deligne-Mumford-Knudson operad $\DMK$ (with $\DMK_n = \bar{\M}_{0,n+1}(\cc)$) is homotopy equivalent to the ``homotopy trivialization'' of the sub-operad of $1\to 1$ operations in $\FLD$, i.e. to the homotopy pushout (in a standard model structure on operads) of the diagram 
\[
\begin{tikzcd}
S^1\arrow[r]\arrow[d]&\FLD\\
\pt, &
\end{tikzcd}
\]
where in the left hand column for $G$ a group (more generally, a monoid), we interpret $G$ as an operad by taking 
$$G_{n\to 1} : = \begin{cases}
G, & n = 1\\
\emptyset, &\t{else.}
\end{cases}
$$
This implies an analogous pushout diagram on the level of chains. Since the operad $\DMK$ is algebro-geometric and defined over $\qq$ (indeed, also over $\zz$), there is an evident de Rham lattice $C_*^{dR}(\DMK)$ (dual to rational de Rham cochains) in $C_*(\DMK, \cc).$ Techniques of \cite{dc} and \cite{oancea_vaintrob} can be extended to a log geometric context to show that $C_*^{dR}(\DMK,\qq)$ in fact fits in a canonical pushout diagram 
\[
\begin{tikzcd}
C_*^{dR}(\gg_m,\qq)\arrow[r]\arrow[d] & C_*^{dR}(\fldlog,\qq)\arrow[d]\\
\qq \arrow[r]&     \arrow[ul, phantom, "\lrcorner", very near start]
C_*^{dR}(\DMK)
\end{tikzcd}
\]
There is an explicit combinatorial computation of pushouts of derived operads (using the Boardman-Vogt resolution), and this seems to be the first explicit construction of a quasiisomorphism between $C_*^{dR}(\DMK)$ and an explicitly constructed combinatorial operad. The dual of the resulting quasiisomorphism is compatible with product structure, and it seems to be a new result to give an explicit combinatorial dg model (i.e., quasiisomorphic via explicit maps) for the cdga's $C^*(\bar{\M}_{0, n+1}, \qq)$ which is directly compatible with all boundary maps.


\section{Appendix}\label{sec:appendix}\label{appendix}
In this appendix we prove Theorem \ref{bettitodr}. In fact we prove it in a slightly generalized context, with normal-crossings schemes replaced by $\Omega$-smooth varieties. In particular, we discuss the notion of $\Omega$-smoothness in section \ref{sec:dsmooth}. Our proof is a direct transposition of the proof in Kato and Nakayama's \cite{kato-nakayama} into an $\infty$-categorical language. The main theorem is as follows.
\begin{thm}\label{bettitodr2}
Let $\t{Acyc}_{NC}$ be the category of acyclic normal-crossings log schemes. Then there exists a functorial and symmetric monoidal chain of quasiisomorphisms between the functors $C^*_{sing}$ and $C^*_{dR}.$ 
\end{thm}
We assume acyclicity here because it is more complicated to give a functorial model for the de Rham cochains functor for non-acyclic log schemes as $H^0(X, \Omega^0_{log})\to H^0(X, \Omega^1_{log})\to \ldots.$ On the infinity-categorical level, acyclicity can be dropped and follows from the above theorem by gluing, as any log variety can be glued out of acyclic ones (since a log variety with affine underlying scheme is acyclic).

\subsection{$\Omega$-smoothness.}\label{sec:dsmooth}
This section serves as a notational disambiguation between the notion of smoothness in this paper and classical results about log varieties. We give a definition of the notion of $\Omega$-smoothness, a generalization of the notion of smoothness for log varieties. The notion was introduced in \cite{kato-nakayama}, but was not named there (rather, it is given as a technical condition for the Betti-to de Rham comparison result). We assume knowledge of standard log geometry, for example after \cite{ogus_logbook}. We give a definition in terms of toric varieties, and compare it to Ogus' notion of smoothness of \emph{idealized} log schemes.

Let $X_{\Delta}$ be an affine toric variety (with $\Delta$ a rational affine cone) and $X_{\Delta'}$ a toric subvariety corresponding to a face $\sigma$ of $\Delta,$ with $\Delta' = \Delta/\sigma$. We define $\X_{\Delta}$ to be the standard log structure on $\X$ (this is the initial log variety whose sheaf $L$ of monoids admits a map to the constant semigroup $M_{\Delta^\vee}$ of characters associated to $X$). We define the \emph{relative log toric variety}
\[\X_{\Delta, \Delta'}: = \X_{\Delta}\times_{X_\Delta} X_{\Delta'}.\] 

\begin{defi}\label{def:omega_smooth}
A log variety $\X$ with underlying space $X$ is called $\Omega$-smooth if it is \'etale locally modeled on $\X_{\Delta,\Delta'},$ i.e., for any geometric point $x\in X(\bar{k})$ there is an \'etale neighborhood $U$ of $x$ with $U$ isomorphic to an open in the relative log toric variety $\X_{\Delta, \Delta'}.$
\end{defi}
\begin{example}
The point $\pt_{log}$ is $\Omega$-smooth, since it is isomorphic to the relative toric log variety $\X_{\rr_{\ge 0}, \{0\}}$. 
\end{example}

From the point of view of differential forms, de Rham theory and $D$-modules, $\Omega$-smoothness is ``just as good'' (i.e., has the same good finite resolution properties) as log smoothness. In particular it is for this class of varieties that Kato and Nakayama proved their seminal Betti to de Rham comparison result in Theorem 0.2 (2) of \cite{kato-nakayama}. Kato and Ogus \cite{ogus_logbook} after him do not call such varieties \emph{smooth}, as this notion is for them a deformation-theoretic property, that holds only for log varieties locally modeled on toric varieties of type $\X_\Delta$ (in particular, with generically trivial log structure). 

However it is possible to relate $\Omega$-smoothness to a deformation-theoretic notion of smoothness. To express this in the language of Ogus' \cite{ogus_logbook}, we need to introduce as an intermediate step a further enlargement of the category of log schemes.

Let $\X$ be a log scheme with sheaf of functions $\oo,$ log monoid sheaf $\L$ (an \'etale sheaf of commutative unital monoids) and log exponent map $\alpha: \L\to (\oo, \cdot).$ 
\begin{defi}
An idealized structure on $\X$ is a choice of (not necessarily unital) monoid subsheaf $\I\subset \L$ such that
\begin{itemize}
\item $\I$ is an ideal, i.e., $\I \L \subset \I$
\item $\alpha(\I) = 0$.
\end{itemize}
\end{defi}

A log scheme $\X$ admits two canonical idealized structures: a trivial idealized structure, $\I = \emptyset$ (the terminal object, with the smallest ideal) and a \emph{tautological} idealized structure, $\I = \alpha^{-1}(0)$ (the initial object, with largest ideal). We will view ordinary log schemes as a subcategory of all log schemes via the latter functor (sending a log scheme to itself with tautological idealized structure). 

Now to Ogus, smoothness of a map $f:X\to Y$ is a liftability criterion for log thickenings $i:S\to T$ over $f.$ A log thickening is a certain kind of closed immersion $\X\to \Y$ of log schemes such that the map of underlying schemes $X\to Y$ is an infinitesimal thickening in the ordinary sense. For example, the embedding \[i:\pt_{log}\to \dd^1_{log},\] is (ind-) log formal for $\dd^1_{log} : = (\aa^1, 0)_{log}\times_{\aa^1}\dd^1$ is the log formal disk. In particular, the map $\pt_{log}\to \pt$ is not smooth in this sense since there is no way to fill in the dotted arrow in the diagram
\[\begin{tikzcd}
\pt_{log}\arrow[d, "id"]\arrow[r, "i"]&\dd^1_{log}\arrow[d]\arrow[dl, dashed]\\
\pt_{log}\arrow[r]&\pt
\end{tikzcd}\]

Now in the context of idealized log varieties, Ogus introduces a notion of \emph{idealized} log thickening which is more restrictive, in the sense that certain log thickenings of toric varieties cease to be idealized log thickenings with the tautological idealized structure. In particular, the map $\pt_{log}\to \dd_{log}$ is not an idealized thickening. He defines idealized smooth maps to be suitable maps of idealized log schemes which have the lifting condition for idealized thickenings. Since it is more restrictive for a map of log schemes (with tautological idealized structure) to be an idealized thickening, it is not surprising that it is easier for a map of log schemes (with tautological idealized structure) to be smooth. The log point and more general relative toric log varieties are all smooth as idealized log varieties. Conversely Ogus proves the following lemma, in terms of our definition \ref{def:omega_smooth}. 
\begin{lm}
A log scheme $\X$ is smooth as an idealized scheme (with tautological idealized structure) if and only if it is $\Omega$-smooth.
\end{lm}
\begin{proof}\label{ogus_omega}
This is Variant 3.3.5 or Corollary 3.3.4 of \cite{ogus_logbook} for $Y = \pt.$
\end{proof}
In particular, it follows that any normal-crossings scheme is $\Omega$-smooth. All results cited in section \ref{sec:log_disks} for $\Omega$-smooth $\X$ follow from corresponding results of \cite{ogus_logbook} for idealized smooth schemes.

\subsection{Betti and de Rham models and comparisons.}\label{sec:betti_to_dr}

Here we prove the Betti-to-de Rham comparison result, Theorem \ref{bettitodr2}, by repeating Kato and Nakayama's proof in \cite{kato-nakayama} with $\infty$-categorical language.

Let $\cat$ be the category of (small, presentable) $\infty$-categories. Let $\Shf^{\top}$ be the category obtained by the Grothendieck construction for the functor $\Shf:\top\to \cat$ which takes a topological space $T$ to the $\infty$-category of complexes of sheaves on $T$ (localized at quasiisomorphisms). Objects of $\Shf^{\top}$ are pairs $(T, \F)$ where $T$ is a topological space and $\F$ is a sheaf of complexes of vector spaces, and morphisms $\hom\big((T, \F), (T', \F')\big)$ are $$\bigsqcup_{\pi\in\t{Maps}(T', T)}\hom(\pi^*(\F), \F')$$ where $\pi$ ranges over the \emph{set} of maps of topologies and $\hom(\pi^*(\F), \F')$ is the $\infty$-groupoid of maps between complexes of sheaves. Note that if we impose suitable finiteness conditions on the topological space $T$, which are satisfied for all topologies we consider, this category can be obtained from the ordinary Grothendieck construction $1$-category of pairs $(T, \F)$ localized at maps $\pi:T'\to T, f:\pi^*\F\to \F'$ such that $\pi$ is a homeomorphism and $f$ a quasiisomorphism.

The category $\Shf^{\top}$ is symmetric monoidal, with symmetric monoidal structure given by $(T, \F)\times (T', \F') := (T\times T', \F\boxtimes \F').$ The functor of (derived) global sections gives a symmetric monoidal functor $R\Gamma: \Shf^{\top}\to D^b\vect$ to the category of complexes of vector spaces. (This functoriality of derived global sections is where $\infty$-categorical language gets used in an essential way.)

We will construct the Betti-to de Rham comparison by associating, in a functorial and symmetric monoidal way, four objects in $\Shf^{\top}$ to any log scheme, and relating them with maps that induce quasiisomorphisms on global sections.

The four pairs we consider are as follows.
\begin{itemize}
\item $(X^{alg},\Omega_{alg}^*)$, the Zariski topology on the classical scheme $X$ underlying $\X$, with complex of sheaves $\Omega^*_{alg}$ the algebraic log de Rham cohomology; 
\item $(X^{an}, \Omega_{an}^*)$, the analytic topology on the complex space $X(\cc),$ with complex of sheaves $\Omega^*_{an},$ the analytic log de Rham cohomology; 
\item $(\X^{top},\Omega_{top}^*)$ the standard topology on the Kato-Nakayama space $\X^{top} = \X(\cc_{KN}),$ with complex of sheaves $\Omega^*_{top}$ the complex defined by Kato and Nakayama in \cite{kato-nakayama} (there denoted $\omega_{log}$), and finally,
\item $(\X^{top}, \underline{\cc}),$ the constant sheaf on $\Omega^{top}.$
\end{itemize}
Note that $R\Gamma(X^{alg}, \Omega_{alg}^*)$ computes the algebraic de Rham cohomology (and agrees with our explicit model when $\X$ is acyclic) and the last $R\Gamma(\X^{top}, \underline{\cc})$ is quasiisomorphic to $C^*_{sing}(\X^{top})$ via a standard symmetric monoidal comparison, since $\X^{top}$ is locally contractible (for example, since it is homeomorphic to a finite CW complex).

These pairs, viewed as objects of the Grothendieck category $\Shf^{\top},$ are related by a sequence of maps as follows: 
\begin{align}\begin{split}\label{eq:quism}
(\X^{top},\underline{\cc})&\xrightarrow{(\one, i_{const})} \\ (\X^{top}, \Omega^*_{top})&\xleftarrow{(\tau, i^{top,an})} (X^{an}, \Omega^*_{an})\xleftarrow{(g, i^{an, alg})} (X^{alg}, \Omega^*_{alg}).
\end{split}\end{align}
The map $(\X^{top},\underline{\cc})\to (\X^{top}, \Omega^*_{top})$ is the identity on the topological space, with $\underline{\cc}$ mapping to $\Omega^*_{top}$ as constant functions of $\oo_{top} = \Omega_{top}^0$. It is a quasiisomorphism of complexes of sheaves by the so-called Kato-Nakayama Poincar\'e lemma, \cite{kato-nakayama} Theorem 3.8. (This in fact was what motivated the construction of $\X^{top}$.) The map of topologies $$\tau:\X^{top}\to X^{an}$$ is the Kato-Nakayama realization of the canonical projection $\X\to X$ and the map $X^{an}\to X^{alg}$ is the standard GAGA comparison which pulls back Zariski opens in $X^{alg}$ to the corresponding topological opens in $X^{an}$. The maps relating sheaves of differential forms are then defined in a standard way, using the standard basechange identifications, \[\Omega^*_{an} \cong g^*\Omega^*_{alg}\otimes_{g^*\oo_{alg}}\oo_{an}\] and
\[\Omega^*_{top} \cong \tau^*\Omega^*_{an}\otimes_{\tau^*\oo_{an}}\oo_{top}.\]

It is proven in \cite{kato-nakayama} (see also \cite{ogus_logbook}, Theorems V.1.4.7 and V.4.2.5) that the maps $(\tau, i^{top, an})$ and $(g, i^{an, alg})$ both induce adjoint natural transformations which are quasiisomorphisms to the derived pushforward, hence give quasiisomorphisms after taking global sections. For the pair $(\tau, i^{top, an})$ this follows from a certain fiber computation (which reduces to computing the cohomology of a certain local system) and for the pair $(g, i^{an, alg})$ this follows from standard GAGA results about holonomic $D$ modules. All three maps are evidently lax symmetric monoidal in the log scheme $\X$. Thus the maps on derived global sections induced by the three maps of (\ref{eq:quism}) are quasiisomorphisms of functors compatible with symmetric monoidal structure on the category of log schemes. They thus provide a canonical (up to contractible space of choices) quasiisomorphism between the symmetric monoidal DG functors $\Gamma\Omega^*$ and $C^*_{\t{Betti}}$ on the category $\t{Acyc}$ of acyclic log schemes. By standard rectifiability results on $\infty$-operads (see \cite{chh}), they imply also the existence of a canonical (in a model-theoretic sense) chain of symmetric monoidal isomorphisms between the two functors. 
\qed




\end{document}